\documentclass[12pt,a4paper]{amsart}
\usepackage{dsfont}
\usepackage{graphicx}
\usepackage{amssymb}
\usepackage[all]{xy}
\usepackage{amsmath}
\usepackage{tikz}
\usepackage[french,english]{babel}
\usepackage{amssymb,amsmath,amsfonts,amsthm,amscd}
\usepackage{indentfirst,graphicx}
\usepackage[font={scriptsize},captionskip=5pt]{subfig}
\usepackage{tikz}
\usetikzlibrary{matrix}

\newtheorem{theorem}{Theorem}[section]
\newtheorem{corollary}[theorem]{Corollary}
\newtheorem{proposition}[theorem]{Proposition}
\newtheorem{lemma}[theorem]{Lemma}

\newtheorem{definition}[theorem]{Definition}
\newtheorem{setup}[theorem]{Setup}

\newtheorem{remark}[theorem]{Remark}

\newcommand{\cB}{{\mathcal B}}

\newcommand{\cI}{{\mathcal I}}

\newcommand{\cO}{{\mathcal O}}

\newcommand{\cR}{{\mathcal R}}

\newcommand{\cX}{{\mathcal X}}
\newcommand{\cN}{{\mathcal N}}
\newcommand{\cM}{{\mathcal M}}

\newcommand{\cU}{{\mathcal U}}

\newcommand{\cOXp}{{\cO_X^p}}
\newcommand{\bC}{{\mathbb{C}}}
\newcommand{\bP}{{\mathbb{P}}}

\newcommand{\bZ}{{\mathbb{Z}}}
\newcommand{\bv}{\bold{v}}

\usepackage{fancyhdr}
\begin{document}

\title{INFINITESIMAL LIPSCHITZ CONDITIONS ON A FAMILY OF ANALYTIC VARIETIES}
\author{Terence Gaffney and Thiago da Silva}

\maketitle

\begin{abstract}
{\small In this work, we extend the concept of the double of an ideal to the context of modules. We also obtain the genericity of the infinitesimal Lipschitz condition A for an enlarged class of analytic spaces.}
\end{abstract}

\let\thefootnote\relax\footnote{2010 \textit{Mathematics Subjects Classification} 32S15, 14J17, 32B99, 32C15, 13B22
	
	}


\section*{Introduction}

\thispagestyle{empty}

The definition of Lipschitz saturation of an ideal appears in \cite{G2}, in the context of bi-Lipschitz equisingularity. The study of bi-Lipschitz equisingularity was started by Zariski \cite{Z}, Pham and Teissier \cite{PT}, and was further developed by Lipman \cite{L}, Mostowski \cite{M1,M2}, Parusinski \cite{PA1}, Birbrair \cite{B} and others.

In this work we continue the study started in \cite{G2}, which is the study of bi-Lipschitz equisingularity from the perspective of the work on Whitney equisingularity (see \cite{G3}).

The Lipschitz Saturation and the double of an ideal $I$, denoted $I_S$ and $I_D$, respectively, were defined in \cite{G2}, where $I$ is a sheaf of ideals of $\cO_X$, the analytic local ring of an analytic variety $X$. The ideal $I_S$ consists of elements in $\cO_X$ for which the quotient of its pullback by the blowup-map, with a local generator of the pullback of $I$ is Lipschitz. The double $I_D$ is the submodule of $\cO_{X\times X}^2$ generated by $(h\circ\pi_1,h\circ\pi_2)$, $h\in I$, where $\pi_1,\pi_2:X\times X\rightarrow X$ are the projections. Theorem 2.3 of \cite{G2} gives a relation between $I_S$ and the integral closure of $I_D$, and is very useful to get conditions for Bi-Lipschitz equisingularity in a family of curves. In \cite{SGP} the authors use the integral closure of ideals and the double to describe the bi-Lipschitz equisingularity of families of Essentially Isolated Determinantal Singularities.

In section \ref{sec0} we recall some basic background material.

In section \ref{sec1} we develop the idea of the double of a module, getting explicit sets of generators of the double from a known set of generators of the module, working toward an extension of Lemma 2.2 of \cite{G2} to the module setting. This set of generators will be very useful in the proofs of some results, mainly in section \ref{sec2}. We also compute the cosupport of the double, which gives us exactly the locus where it make sense to ask about the infinitesimal Lipschitz conditions, defined on section \ref{sec2}. 

Further we prove  Proposition \ref{P2.13} which states that the stalk of the double of a sheaf of modules $\cM$ at $(x,x')$, $x\ne x'$, is the direct sum of the stalks of $\cM$ at $x$ and $x'$. Thus, the stalk of the double carries the same information as the stalks  of $\cM$ do at $x$ and $x'$, as long $x\neq x'$. If $\cM$ is the jacobian module of a family of analytic varieties, the stalks at $x$ and $x'$ determine the tangent hyperplanes at these two points. Since, to control the Lipschitz behavior of the tangent hyperplanes to $X$, it is natural to look for a sheaf on $X\times X$ whose stalks determine the tangent hyperplanes at each pair of distinct points, it is natural to consider the double of the jacobian module. 

The infinitesimal Lipschitz conditions for hypersurfaces were defined in \cite{G1}. In section \ref{sec2} we extend these definitions to an analytic variety with arbitrary codimension, using the double of a module (the jacobian module), developed in section \ref{sec1}. We prove the $iL_A$ condition is generic and then we apply this to the grassmanian modification of an analytic variety in the section 4.

In section 5, we define a strengthening of the $iL_A$ condition for the case of families of curves, and show this strengthened version implies the Lipschitz equisingularity of families of ICIS curves.

\section*{Acknowledgements}

The authors are grateful to Nivaldo Grulha Jr. for his careful reading of this work, and to Maria A. S. Ruas, for the valuable conversations about the subject of this work. The first author was supported in part by PVE-CNPq, grant 401565/2014-9. The second author was supported by Funda\c{c}\~ao de Amparo \`a Pesquisa do Estado de S\~ao Paulo - FAPESP, Brazil, grant 2013/22411-2.

\section{Background on Lipschitz Saturation of Ideals and Integral Closure of Modules}\label{sec0}

The Lipschitz saturation of a local ring was defined by Pham and  Teissier in \cite{PT}.

\begin{definition}
Let $I$ be an ideal of $\cO_{X,x}$,  $SB_I(X)$ the saturation of the blow-up and $\pi_S: SB_I(X)\rightarrow X$ the projection map. The {\bf Lipschitz saturation} of the ideal $I$ is denoted $I_S$, and is the ideal $I_S:=\{h\in\cO_{X,x}\mid\pi_S^*(h)\in\pi_S^*(I)\}$.
\end{definition}

Since the normalization of a local ring $A$ contains the Lipschitz Saturation of $A$, it follows that $I\subseteq I_S\subseteq\overline{I}$. In particular, if $I$ is integrally closed then $I_S=\overline{I}$.

This definition can be given an equivalent statement using the theory of integral closure of modules. Since Lipschitz conditions depend on controlling functions at two different points as the points come together, we should look for a sheaf defined on $X\times X$. We describe a way of moving from a sheaf of ideals on $X$ to a sheaf on $X\times X$. 

Let $\pi_1,\pi_2: X\times X\rightarrow X$ be the projections to the i-th factor, and let $h\in\cO_{X,x}$. Define $h_D\in\cO_{X\times X,(x,x)}^{2}$ as $(h\circ\pi_1,h\circ\pi_2)$, called the double of $h$. We define the double of the ideal $I$, denoted $I_D$, as the submodule of $\cO_{X\times X,(x,x)}^2$ generated by $h_D$, where $h$ is an element of $I$.

We can see in \cite{G2}, the following result gives a link between Lipschitz saturation and integral closure of modules.

\begin{theorem}[\cite{G2}, Theorem 2.3]
Suppose $(X,x)$ is a complex analytic set germ, $I\subseteq\cO_{X,x}$ and $h\in\cO_{X,x}$.  
\noindent Then $h\in I_S$ if, and only if, $h_D\in\overline{I_D}$. 
\end{theorem}

Using the Lipschitz saturation of ideals (and doubles), in \cite{G1} the first author defined the infinitesimal Lipschitz conditions for hypersurfaces.

Let $X^{n+k},0\subseteq\bC^{n+1+k},0$ be a hypersurface, containing a smooth subset $Y$ embedded in $\bC^{n+1+k}$ as $0\times\bC^k$, with $p_Y$ the projection to $Y$. Assume $Y=S(X)$, the singular set of $X$. Suppose $F$ is the defining equation of $X$, $(z,y)$ coordinates on $\bC^{n+1+k}$. Denote by $f_y(z)=F(z,y)$ the family of functions defined by $F$, and by $X_y$, $f_y^{-1}(0)$. Assume that $X_y$ has an isolated singularity at the origin. Let $m_Y$ denote the ideal defining $Y$, and $J(X)_Y$, the ideal generated by the partial derivatives with respect to the $y$ coordinates, $J_z(X)$, those with respect to the $z$ coordinates.
Here we work with the double relative to $Y$, which means that we work with the projections $\pi_1$ and $\pi_2$ defined on the fibered product $X\underset{Y}{\times}X$. Then $I_D$, the double of $I$ with respect to $Y$, is defined on $X\underset{Y}{\times}X$, similarly to the definition of $I_D$ in the absolute case.

\begin{definition}
We say the pair $(X,Y)$ satisfy the $iL_{m_Y}$ condition at the origin if either of the two equivalent conditions hold:
\begin{enumerate}
\item $J(X)_Y\subseteq(m_YJ_z(X))_S$
\item $(J(X)_Y)_D\subseteq\overline{(m_YJ_z(X))_D}$
\end{enumerate}
\end{definition}
An analogous condition for $iL_{m_Y}$ is $J(X)_Y\subseteq\overline{m_YJ_z(X)}$. This is the equivalent to the Verdier's condition $W$ or the Whitney conditions.
Next we give the definition of $iL_A$.

\begin{definition}
We say the pair $(X,Y)$ satisfy the $iL_{A}$ condition at the origin if either of the two equivalent conditions hold:
\begin{enumerate}
\item $J(X)_Y\subseteq(J_z(X))_S$
\item $(J(X)_Y)_D\subseteq\overline{(J_z(X))_D}$
\end{enumerate}
\end{definition}

The analogous condition is $J(X)_Y\subseteq\overline{J_z(X)}$. If one works on the ambient space, then this is equivalent to the $A_F$ condition.

In Proposition 4.1 of \cite{G1} it is proved that the cosupport of $(m_YJ_z(X))_D$ and $(J_z(X))_D$ on $X\underset{Y}{\times}X$ are equal, and consist of $$\Delta(X)\cup(X\times_{Y}0)\cup(0\times_{Y}X).$$

In \cite{G1} we have the following result.

\begin{theorem}[\cite{G1}, Proposition 4.2 and Theorem 4.3]
In the set-up of this section, the integral closure conditions for $iL_A$ and $iL_{m_Y}$ hold at all points of $\Delta(X)-((0,0)\times Y)$, and both conditions hold at all point of $(0\times_{Y}X)\cup(X\times_{Y}0)-((0,0)\times Y)$ if condition W holds at all point $(0,y)$, $y\in Y$. Furthermore, the $iL_A$ condition holds generically along $(0,0)\times Y$.
\end{theorem}

In Section \ref{sec2} we generalize these results for $X$ of arbitrary codimension.

Let us recall two results about the integral closure of modules which will inspire good definitions for Lipschitz saturation of modules.

\vspace{0,5cm}

{\it The ideal sheaf $\rho(\cM)$ on $X\times \bP^{p-1}$ associated to a submodule sheaf $\cM$ of $\cOXp$ (see \cite{GK})}: Given $h=(h_1,...,h_p)\in\cOXp$ and $(x,[t_1,...,t_p])\in X\times\bP^{p-1}$, with $t_i\neq0$, we define $\rho(h)$ as the germ of the analytic map given by $\sum\limits_{j=1}^p h_j(z)\frac{T_j}{T_i}$ which is well-defined on a Zariski open subset of $X\times\bP^{p-1}$ that contains the point $(x,[t_1,...,t_p])$. We define $\rho(\cM)$ as the ideal generated by $\{\rho(h)\mid h\in \cM\}$. The next result, proved in \cite{GK}, gives a strong relation between the integral closure of modules and ideals.

\begin{proposition}[\cite{GK}, Proposition 3.4]\label{proposition GK}
Let $h\in\cO_{X,x}^p$. Then $h\in\overline{\cM}$ at $x$ if, and only if, $\rho(h)\in\overline{\rho(\cM)}$ at all point $(x,[t_1,...,t_p])\in V(\rho(\cM))$.
\end{proposition}

In \cite{G3} there is another way to make a link between the integral closure of modules and ideals, using minors of a matrix of generators of $\cM$. 

Let $\cM$ be a sheaf of submodules of $\cOXp$, and $[\cM]$ a matrix of generators of $\cM$. For each $k$, let $J_k(\cM)$ denote the ideal of $\cO_X$ generated by the $k\times k$ minors of $[\cM]$. If $h\in\cOXp$, let $(h,\cM)$ be the submodule generated by $h$ and $\cM$.

\begin{proposition}[\cite{G3}, Corollary 1.8]\label{proposition G}
Suppose $(X,x)$ is a complex analytic germ with irreducible components $\{V_i\}$. Then, $h\in\overline{\cM}$ at $x$ if, and only if, $J_{k_i}((h,\cM_i))\subseteq \overline {J_{k_i}(\cM_i)}$ at $x$, where $\cM_i$ is the submodule of $\cO_{V_i,x}^p$ induced from $\cM$ and $k_i$ is the generic rank of $(h,\cM_i)$ on $V_i$. 
\end{proposition}

\section{The double of a Module and basic properties}\label{sec1}

In this section we extend to modules the notion of the double of an ideal, getting some basic properties.

Let $X \subseteq \mathbb{C}^n$ be an analytic space, and let $\cM$ be an $\cO_X $-submodule of $\cO_X^p$. Consider the projection maps $\pi_1,\pi_2: X \times X \rightarrow X$. We assume that $\cM$ is finitely generated by global sections.

\begin{definition}
Let $h\in\cOXp$. The double of $h$ is defined as the element

\noindent $h_D:=(h\circ\pi_1,h\circ\pi_2) \in\cO_{X\times X}^{2p}$.

The double of $\cM$, denoted by $\cM_D$, is defined as the $\cO_{X\times X}$-submodule of $\cO_{X\times X}^{2p}$ generated by $\{h_D \mid  h\in \cM(X)\}$.

\end{definition}

We want to recover some results which are true in the ideal case, i.e, when $p=1$ and $\cM=\cI$ is an ideal sheaf. We start by obtaining a set of generators for $\cM_D$ from a set of generators of $\cM$.

Consider $z_1,...,z_n$ the coordinates on $\mathbb{C}^n$. 

\begin{lemma}\label{L1}
\begin{enumerate}
\item $(\alpha h)_D=-(0,(\alpha\circ\pi_1-\alpha\circ\pi_2)(h\circ\pi_2))+(\alpha\circ\pi_1)h_D$, for all $\alpha\in\cO_X$ and $h\in\cOXp$;

\item $(0,(\alpha\circ\pi_1-\alpha\circ\pi_2)(h\circ\pi_2))\in \cM_D$, for all $h\in \cM$ and $\alpha\in\cO_X$;

\item $\alpha\circ\pi_1-\alpha\circ\pi_2 \in I(\Delta(X))=(z_1\circ\pi_1-z_1\circ\pi_2,\ldots, z_n\circ\pi_1-z_n\circ\pi_2)$, for all $\alpha\in\cO_X$;

\item $(g+h)_D=g_D+h_D$, for all $g,h\in\cOXp$. 
\end{enumerate}
\end{lemma}

\begin{proof}
(1) We have: $(\alpha h)_D=((\alpha\circ\pi_1)(h\circ\pi_1),(\alpha\circ\pi_2)(h\circ\pi_2))\newline=-(0_{\cO_{X\times X}^{p}},(\alpha\circ\pi_1-\alpha\circ\pi_2)(h\circ\pi_2))+(\alpha\circ\pi_1)h_D$.

(2) Since $h\in \cM$ then $\alpha h\in \cM$, so $h_D\in \cM_D$ and $(\alpha h)_D\in \cM_D$. Thus, by (a) we have that $(0_{\cO_{X\times X}^{p}},(\alpha\circ\pi_1-\alpha\circ\pi_2)(h\circ\pi_2))=(\alpha\circ\pi_1)h_D - (\alpha h)_D \in \cM_D$.

(3) Obviously $\alpha\circ\pi_1-\alpha\circ\pi_2$ vanishes on the diagonal of $X$.

(4) Notice that: $(g+h)_D=((g+h)\circ\pi_1,(g+h)\circ\pi_2)=(g\circ\pi_1+h\circ\pi_1,g\circ\pi_2+h\circ\pi_2)=(g\circ\pi_1,g\circ\pi_2)+(h\circ\pi_1,h\circ\pi_2)$.
\end{proof}

The next proposition gives a set of generators of $M_D$, from a known set of generators of $\cM$.

\begin{proposition}\label{P2}
Suppose that $\cM$ is generated by global sections \\ $\{h_1,\ldots,h_r \}$. Then, the following sets are generators of $\cM_D$:
\begin{enumerate}

\item $\cB=\{(h_1)_D,\ldots,(h_r)_D\} \cup \{(0_{\cO_{X\times X}^{p}},(z_i\circ\pi_1-z_i\circ\pi_2)(h_j\circ\pi_2))$ $|$ $i\in\{1,\ldots,n\}$ and $j\in\{1,\ldots,r\}\}$.

\item $\cB'=\{(h_1)_D,\ldots,(h_r)_D\} \cup \{((z_i\circ\pi_1-z_i\circ\pi_2)(h_j\circ\pi_1),0_{\cO_{X\times X}^{p}}))$ $|$ $i\in\{1,\ldots,n\}$ and $j\in\{1,\ldots,r\}\}$.

\item $\cB''=\{(h_1)_D,\ldots,(h_r)_D\} \cup \{(z_ih_j)_D$ $|$ $i\in\{1,\ldots,n\}$ and $j\in\{1,\ldots,r\}\}$.

\end{enumerate}

\end{proposition}

\begin{proof}
(1) Let $N$ be the submodule of $\cO_{X\times X}^{2p}$ generated by $\cB$. By Lemma \ref{L1} (b) we have that $N\subseteq \cM_D$. Now, to verify that $\cM_D\subseteq N$ it is enough to check that $h_D\in N, \forall h\in \cM$. Indeed, if $h\in \cM$ we can write $h= \sum\limits_{j=1}^{r} \alpha_j h_j $, for some $\alpha_j \in \cO_X$. By Lemma \ref{L1} (a) and (d) we have that $$h_D= \left( \sum\limits_{j=1}^{r} (\alpha_{j}\circ\pi_1)(h_j)_D  \right) - \left(\sum\limits_{j=1} ^{r} (0_{\cO_{X\times X}^{p}},(\alpha_j\circ\pi_1-\alpha_j\circ\pi_2)(h_j\circ\pi_2)) \right)$$ Clearly the first sum is in $N$. By Lemma \ref{L1} (c) we have that each $\alpha_j\circ\pi_1-\alpha_j\circ\pi_2$ belongs to the ideal $I(\Delta(X))$, so the second sum is in $N$.

(2) This is completely analogous to item (1). 

(3) We use (1). Let $N$ be the submodule of $\cO_{X\times X}^{2p}$ generated by $\cB''$. For all $j\in\{1,...,r\}$ and $i\in\{1,...,n\}$ we have $$(z_ih_j)_D=(z_i\circ\pi_1)(h_j)_D - (0_{\cO_{X\times X}^{p}},(z_i\circ\pi_1-z_i\circ\pi_2)(h_j\circ\pi_2))\in M_D,$$ by previous lemma. Hence, $N\subseteq \cM_D$. Now, to check that $\cM_D\subseteq N$, it is enough to verify that all the generators of $\cM_D$ given in (1) are in $N$. We already have $(h_j)_D\in N$, for all $j\in\{1,...,r\}$. Also, for all $j$ and $i$ we have $$(0_{\cO_{X\times X}^{p}},(z_i\circ\pi_1-z_i\circ\pi_2)(h_j\circ\pi_2))=(z_i\circ\pi_1)(h_j)_D-(z_ih_j)_D\in N.$$
\end{proof}

We can develop the notion of the double in the family case. Suppose that $\cX\subseteq \mathbb{C}^{n+k}$ is an analytic space and let $Y=0\times \bC^k\subseteq \cX$. Identifying $Y=0\times \bC^k=\bC^k$ we have that $\cX\subseteq \bC^{n}\times Y$. Let $p:\cX\subseteq \bC^{n}\times Y \longrightarrow Y$ be the projection, $\cX\underset{Y}{\times}\cX$ the fibered product, with the projections maps $\pi_1,\pi_2: \cX\underset{Y}{\times}\cX \rightarrow \cX$.

Let $h\in\cO_{\cX}^p$. The double of $h$ relative to $Y$ is defined by $$h_{D,Y}:=h_D:=(h\circ\pi_1,h\circ\pi_2)\in \cO_{\cX\underset{Y}{\times}\cX}^{2p}.$$

The double of a submodule $\cM$ of $\cO_{\cX}^p$ relative to $Y$ is defined as the $\cO_{\cX\underset{Y}{\times}\cX}$-submodule of $\cO_{\cX\underset{Y}{\times}\cX}^{2p}$ generated by $\{h_D  \mid h\in \cM(\cX)\}$, and is denoted by $\cM_D$ (or $\cM_{D,Y}$).

Let $z_1,\ldots,z_{n},y_1,\ldots,y_k$ be the coordinates on $\bC^{n+k}$. It is easy to see that Lemma \ref{L1} still holds when we are working with the projections restricted to the fibered product $\cX\underset{Y}{\times}\cX$, and since each $y_{\ell}\circ\pi_1-y_{\ell}\circ\pi_2$ vanishes on the fibered product, then we get the following analogous proposition.

\begin{proposition}\label{P3}
Suppose that $\cM$ is generated by $\{h_1,\ldots,h_r \}$. Then, the following sets are generators of $\cM_D$ relative to $Y$:
\begin{enumerate}

\item $\cB=\{(h_1)_D,\ldots,(h_r)_D\} \cup \{(0_{\cO_{X\times X}^{p}},(z_i\circ\pi_1-z_i\circ\pi_2)(h_j\circ\pi_2))$ $|$ $i\in\{1,\ldots,n\}$ and $j\in\{1,\ldots,r\}\}$.

\item $\cB'=\{(h_1)_D,\ldots,(h_r)_D\} \cup \{((z_i\circ\pi_1-z_i\circ\pi_2)(h_j\circ\pi_1),0_{\cO_{X\times X}^{p}}))$ $|$ $i\in\{1,\ldots,n\}$ and $j\in\{1,\ldots,r\}\}$.

\item $\cB''=\{(h_1)_D,\ldots,(h_r)_D\} \cup \{(z_ih_j)_D$ $|$ $i\in\{1,\ldots,n\}$ and $j\in\{1,\ldots,r\}\}$.

\end{enumerate}

\end{proposition}

In the next proposition we compute the generic rank of the double of a module.
\newpage
\begin{proposition}\label{T2.9}
	Let $(X,x)$ be an irreducible analytic complex germ of dimension $d\ge 1$, and $\cM\subseteq\cO_{X}^p$ a submodule of generic rank $k$. Then $\cM_D$ has generic rank $2k$. 
\end{proposition}

\begin{proof}
	Let $\{h_1,...,h_r\}$ be a set of generators of $\cM$, $[\cM]$ the matrix whose columns are the $h_i$. Let $U$ be the Zariski open and dense subset of $X$ on which the rank of $[\cM]$ is $k$. We use the generators of $\cM$ to construct generators of $\cM_D$ of type 
$(1)$ in \ref{P2}, $[\cM_D]$ denoting the matrix of generators so constructed.
	
	Let $(x_1,x_2)\in U\times U$ off the diagonal. Since $x_1\neq x_2$, for some $i$, $z_i\circ\pi_1-z_i\circ\pi_2\neq 0$ at $(x_1,x_2)$. Then the matrix $[\cM_D(x_1,x_2)]$ has a lower right block which is a non-zero scalar multiple of $[\cM(x_2)]$. Using column operations we can reduce $[\cM_D(x_1,x_2)]$ to a matrix with a lower right block consisting of $p$ rows and $k$ columns of rank $k$ and the rest of the $p$ rows with zero entries. We can then use column operations again to reduce the first $p$ rows to another $p\times k$ block of rank $k$. The non-zero entries of the reduced matrix form a $2p\times 2k$ matrix made up of two blocks of rank $k$, with zeroes above and below them. Hence the reduced matrix has rank exactly $2k$. Since $U\times U-\Delta(U)$ is a Zariski open and dense subset of $X\times X$, the generic rank of $\cM_D$ is $2k$ at every point.
\end{proof}

It is easy to see that in the case when the dimension of $(X,x)$ is zero, the double of $\cM$ is isomorphic to $\cM$, therefore the generic rank does not change.

\begin{corollary}\label{C2.10}
	Let $\{V_i\}$ be the irreducible components of $(X,x)$. For each $i$, if $\cM$ has generic rank $k_i$ on $V_i$ then $\cM_D$ has generic rank $2k_i$ on $V_i\times V_i$. In particular, if $\cM$ has generic rank $k$ on each component of $X$ then $\cM_D$ has generic rank $2k$ on each component of $X\times X$.
\end{corollary} 

Suppose the generic rank of $\cM$ is $k$; let $$\Sigma(\cM):=\{x\in X \mid rank[\cM(x)]<k\}.$$

In the next proposition we compute $\Sigma(\cM_D)$ in $\cO_{X\times X}^{2p}$.

\begin{proposition}\label{P4} Let $\cM$ be a sheaf of submodules of $\cOXp$ of generic rank $k$. Then
	$$\Sigma(\cM_D)= \Delta(X)\cup(X\times \Sigma(\cM))\cup(\Sigma(\cM)\times X).$$ 
\end{proposition}

\begin{proof}
	By Proposition \ref{T2.9} we know $[\cM_D]$ has constant rank $2k$ on $(U\times U)-\Delta(U)$, $U:=X-\Sigma(\cM)$. Hence $\Sigma(\cM_D)$ lies in $\Delta(X)\cup(X\times \Sigma(\cM))\cup(\Sigma(\cM)\times X)$. 
	
	Suppose $(x_1,x_2)\in\Delta(X)$. Then $[\cM_D(x_1,x_2)]$ is a matrix of two identical $p\times n$ blocks, and the kernel vectors of the top block are also in the kernel of the bottom block, so the rank of $[\cM_D(x_1,x_2)]$ is equal to the rank of $[\cM(x_1)]$ which is $\ell\le k< 2k$. Thus, $(x_1,x_2)$ is in $\Sigma(\cM_D)$.
	
	Now suppose $(x_1,x_2)\in\Sigma(\cM)\times X$, $x_1\neq x_2$. Then the matrix $[\cM_D(x_1,x_2)]$ reduces as in the proof of Proposition \ref{T2.9} to a matrix with two blocks, the top left block of size $p\times k_1$, $k_1=\mbox{rank}[\cM(x_1)]$, $k_1<k$ and a bottom right block of size $p\times k_2$, of rank $k_2\le k$.   So the whole matrix has rank $k_1+k_2<2k$.
	
	A similar proof works in the case where $(x_1,x_2)\in X\times\Sigma(\cM)$.
\end{proof}
 
 It is easy to see this proposition still holds in the family case, by taking $X\times_{Y} \Sigma(\cM)$ and $\Sigma(\cM)\times_{Y}X$.
 
 The next proposition generalizes Corollary 3.4 of \cite{G1} for modules.
 
 \begin{proposition}\label{P10}
 Let $\cM\subseteq \cN\subseteq\overline{\cM}$ be $\cO_X$-submodules of $\cOXp$, with $X$ equidimensional. Suppose that $\cM_D$ has finite colength in $\cN_D$ and $\cN_D$ has finite colength in $(\overline{\cM})_D$. Then $$e(\cM_D,(\overline{\cM})_D)=e(\cN_D,(\overline{\cM})_D)\mbox{ if and only if } \overline{\cM_D}=\overline{\cN_D}.$$
 \end{proposition} 
 
 \begin{proof}
 By the principle of additivity \cite{KT}, we have that $$e(\cM_D,(\overline{\cM})_D)=e(\cM_D,\cN_D)+e(\cN_D,(\overline{\cM})_D).$$
 
 Notice that all these multiplicities are well-defined by hypothesis. So, $e(\cM_D,(\overline{\cM})_D)=e(\cN_D,(\overline{\cM})_D)$ if, and only if, $e(\cM_D,\cN_D)=0$, which is equivalent to the equality $\overline{\cM_D}=\overline{\cN_D}$, since $X$ is equidimensional (see \cite{KT}).
 \end{proof}
 
 The following proposition and corollary are useful to make a relation between the saturation and the double of a module, and to work with the infinitesimal Lipschitz conditions.
 
 \begin{proposition}\label{P12}
 Let $h\in\cOXp$. 
 \begin{enumerate}
 \item If $h_D\in\overline{\cM_D}$ at $(x,x')$ then $h\in\overline{\cM}$ at $x$ and $x'$. 
 \item If $h_D\in (\cM_D)^{\dagger}$ at $(x,x')$ then $h\in \cM^{\dagger}$ at $x$ and $x'$.
 \end{enumerate}
 The same result still holds in the family case.
 \end{proposition}  
 \begin{proof}
(1) Let us prove that $h\in\overline{\cM}$ at $x$ (the case at $x'$ is completely analogous). Let $\phi:(\bC,0)\rightarrow(X,x)$ be an arbitrary analytic curve. Define $\gamma:(\bC,0)\rightarrow(X\times X,(x,x'))$ given by $\gamma(t)=(\phi(t),x')$. Since $h_D\in \overline{\cM_D}$ then $h_D\circ\gamma\in \cM_D\circ\gamma$, so we can write $$h_D\circ\gamma=\sum\alpha_j((g_j)_D\circ\gamma)$$ with $g_j\in \cM$ and $\alpha_j\in\cO_{\bC,0}$. Since $\pi_1\circ\gamma=\phi$, comparing the first $p$ coordinates of the above equation, we get $h\circ\phi=\sum\alpha_j(g_j\circ\phi)\in \cM\circ\phi$. Therefore, $h\in\overline{\cM}$ at $x$.

(2) The proof is completely analogous to item (a), working on the strict integral closure.
\end{proof}

The proof in the family case is also analogous, working on the fibered product $X\underset{Y}{\times}X$.

\begin{corollary}\label{C13}
Let $\cM$ and $N$ be $\cO_X$-submodules of $\cOXp$.
\begin{enumerate}
\item If $\cM_D\subseteq\overline{N_D}$ at $(x,x)$ then $\cM\subseteq\overline{\cN}$ at $x$;
\item If $\cM_D\subseteq(\cN_D)^{\dagger}$ at $(x,x)$ then $\cM\subseteq\ \cN^{\dagger}$ at $x$.
\end{enumerate}
The same result still holds in the family case.
\end{corollary}

In next proposition we prove that the integral closure of modules commutes with finite direct sum of modules.

\begin{proposition}\label{P2.13}
	Let $\cM\subseteq\cO_X^p$ be a sheaf of submodules. Consider $(x,x')\in X\times X$ with $x\neq x'$. Then:
	\begin{enumerate}
		\item[a)] $\cM_D=(\cM_x\circ\pi_1)\oplus (\cM_{x'}\circ\pi_2)$ at $(x,x')$;
		
		\item[b)]  $\overline{\cM_D}=(\overline{\cM_x}\circ\pi_1)\oplus (\overline{\cM_{x'}}\circ\pi_2)$ at $(x,x')$. 
	\end{enumerate}
The same result still holds in the family case.
\end{proposition}

\begin{proof}
	 Since $x\neq x'$ then $z_{\ell}\circ\pi_1-z_{\ell}\circ\pi_2$ is an invertible element of $\cO_{X\times X,(x,x')}$, for some $\ell\in\{1,...,n\}$.
	
	(a) Given $h\in\cM_x$ arbitrary, Lemma \ref{L1} implies that \\ $((z_{\ell}\circ\pi_1-z_{\ell}\circ\pi_2)(h\circ\pi_1),0)\in\cM_D$. Since $z_{\ell}\circ\pi_1-z_{\ell}\circ\pi_2$ is invertible then $(h\circ\pi_1,0)\in\cM_D$. Thus, $(\cM_x\circ\pi_1)\oplus 0\subseteq \cM_D$ at $(x,x')$. Analogously, $0\oplus(\cM_{x'}\circ\pi_2)\subseteq \cM_D$ at $(x,x')$. Hence $(\cM_x\circ\pi_1)\oplus (\cM_{x'}\circ\pi_2)\subseteq \cM_D$ at $(x,x')$. The other inclusion is obvious.
	
	(b) Since $z_{\ell}\circ\pi_1-z_{\ell}\circ\pi_2$ is an invertible element of $\cO_{X\times X,(x,x')}$ then using the curve criterion it is easy to see that $(h\circ\pi_1,0)\in\overline{\cM_D}$, $\forall h\in\overline{\cM_x}$. Thus, $(\overline{\cM_x}\circ\pi_1)\oplus 0\subseteq \overline{\cM_D}$ at $(x,x')$. Analogously, $0\oplus(\overline{\cM_{x'}}\circ\pi_2)\subseteq \overline{\cM_D}$ at $(x,x')$. Hence $(\overline{\cM_x}\circ\pi_1)\oplus (\overline{\cM_{x'}}\circ\pi_2)\subseteq \overline{\cM_D}$ at $(x,x')$. 
	
	Proposition \ref{P12} (a) implies the other inclusion.
\end{proof}

\begin{corollary}\label{C2.14}
	Let $\cM\subseteq\cO_X^p$ be a sheaf of submodules. Let $(x,x')\in X\times X$ with $x\neq x'$. 
	
	Then $\overline{\cM_x\circ\pi_1}=\overline{\cM_x}\circ\pi_1$ and  $\overline{\cM_{x'}\circ\pi_2}=\overline{\cM_{x'}}\circ\pi_2$ at $(x,x')$.
\end{corollary}

\begin{proof}
	Item (a) of the previous proposition implies that \\$\overline{\cM_D}=\overline{(\cM_x\circ\pi_1)\oplus (\cM_{x'}\circ\pi_2)}=\overline{(\cM_x\circ\pi_1)}\oplus\overline{ (\cM_{x'}\circ\pi_2)}$ at $(x,x')$. In the other hand, item (b) gives the equation $\overline{\cM_D}=(\overline{\cM_x}\circ\pi_1)\oplus (\overline{\cM_{x'}}\circ\pi_2)$ at $(x,x')$.
\end{proof}

\begin{corollary}\label{C2.13}
	Let $\cM\subseteq\cO_X^p$ be a sheaf of submodules. Consider $(x,x')\in X\times X$ with $x\neq x'$. Then $(\overline{\cM})_D=\overline{\cM_D}$ at $(x,x')$.
\end{corollary}

\begin{proof}
	Using the previous results we have \\ $(\overline{\cM})_D=(\overline{\cM_x}\circ\pi_1)\oplus (\overline{\cM_{x'}}\circ\pi_2)=(\overline{\cM_x\circ\pi_1})\oplus (\overline{\cM_{x'}\circ\pi_2})\\=\overline{(\cM_x\circ\pi_1)\oplus (\cM_{x'}\circ\pi_2)}=\overline{\cM_D}$.
\end{proof}

\newpage

\begin{proposition}\label{P2.14}
	Let $\cM\subseteq\cO_X^p$ be a sheaf of submodules. Suppose $X^d$ is reduced and equidimensional, and $\Sigma(\cM)\subseteq\{0\}$. Then the multiplicity of the pair $e(\cM_D,(\overline{\cM})_D)$ is well-defined at $(0,0)$. 
	
	In particular, in the notation of \cite{G5}, if $\cM$ has finite colength in $\cOXp$ then $H_{2d-1}(\cM_D)=(\overline{\cM})_D$. 
\end{proposition}

\begin{proof}
	We need to show that $\overline{(\overline{\cM})_D}=\overline{\cM_D}$ at any point $(x,x')\neq(0,0)$.
	
	Suppose first $x\neq x'$. So by Corollary \ref{C2.13} one has $(\overline{\cM})_D=\overline{\cM_D}$ at $(x,x')$ which implies that $\overline{(\overline{\cM})_D}=\overline{\overline{\cM_D}}=\overline{\cM_D}$ at $(x,x')$.
	
	Now, we may assume $x=x'$. Since $(x,x)\neq(0,0)$ then $x\neq 0$, i.e, $x\in X-\Sigma(\cM)$. Proposition 1.7 of \cite{G5} implies that $\overline{\cM}=\cM$ at $x$. Thus, taking the double at $(x,x)$ we have $(\overline{\cM})_D=\cM_D$ at $(x,x)$ which implies $\overline{(\overline{\cM})_D}=\overline{\cM_D}$ at $(x,x)$.
\end{proof}

Let $F:(\bC^n,0)\rightarrow(\bC^p,0)$ be an analytic map, $X=F^{-1}(0)$, $d=\dim X$ and $JM(X)$ the jacobian module defined on $X$. Denote by $\Sigma(X)$ the singular set of $X$. The next result is a straightforward consequence of Proposition \ref{P2.14} and the inclusion $\Sigma(JM(X))\subseteq\Sigma(X)$.

\begin{corollary}\label{C2.18}
 If $X$ has isolated singularity at the origin then the multiplicity of the pair of modules $e((JM(X))_D,(\overline{JM(X)})_D)$ is well defined at $(0,0)$.
\end{corollary}

Proposition \ref{P2.13} provides additional motivation for the idea of the double: In order to control the Lipschitz behavior of pairs of tangent planes at two different points $x$ and $x'$ of a family $\cX$, it is helpful to have each module which determines the tangent hyperplanes at each point as part of the construction. Furthermore, this proposition shows that $JM(\cX)_D$ at $(x,x')$ contains both $JM(\cX)_x$ and $JM(\cX)_{x'}$.


\section{The Infinitesimal Lipschitz conditions $iLA$ and $iLm_Y$}\label{sec2}
 
 Now we use some of the results presented in last section to recover some properties about the infinitesimal Lipschitz conditions for the following more general setup.

\begin{setup}\label{setup2.1}
Let $(\cX,0)\subseteq (\bC^{n+k},0) $ be the germ of the analytic space defined by an analytic map $F: \bC^{n}\times\bC^{k}\rightarrow \bC^p$, $n\geq p$, 
$Y=\bC^{k}=0\times\bC^{k}\subseteq \cX$.
Let $F_1,\dots,F_p: \bC^{n}\times\bC^{k}\rightarrow \bC$ be the coordinates functions of $F$, for each $y\in Y$ let $f_y: \bC^{n}\rightarrow \bC^p$ given by $f_y(z):=F(z,y)$ and let $\cX_y:=f_y^{-1}(0)$. Let $z_1,\ldots,z_{n},y_1,\ldots,y_k$ be the coordinates on $\bC^{n+k}$, let $m_Y$ be the ideal of $\cO_{\cX}$ generated by $\{z_1,\ldots,z_{n}\}$, let $JM(\cX)$ be the Jacobian module of $\cX$, let $JM(\cX)_{Y}$ be the module generated by $\{ \frac{\partial F}{\partial y_1},\dots,\frac{\partial F}{\partial y_k}\}$ and let $J_{z}M(\cX)$ be the module generated by  $\{\frac{\partial F}{\partial z_1},\dots,\frac{\partial F}{\partial z_{n}}\}$.
\end{setup}
 
 In this section we work with the double relative to $Y$ and with the projections $\pi_1,\pi_2: \cX\underset{Y}{\times}\cX\rightarrow \cX$.
 
 \begin{definition}
 \begin{itemize}
 \item The pair $(\cX_0,Y)$ satisfy the $iL_{m_Y}$ condition at $(y,0\times 0)\in \cX\underset{Y}{\times}\cX$ if $(JM(\cX)_Y)_D\subseteq\overline{(m_YJ_zM(\cX))_D}$ at $(y,0\times 0)$;
 \item The pair $(\cX_0,Y)$ satisfy the $iL_{A}$ condition at $(y,0\times 0)\in \cX\underset{Y}{\times}\cX$ \newline if $(JM(\cX)_Y)_D\subseteq\overline{(J_zM(\cX))_D}$ at $(y,0\times 0)$.
  \end{itemize}
 \end{definition}
 
 Notice that $iL_{m_Y}$ implies $iL_A$. 

\begin{lemma}\label{L8}
\begin{enumerate}
\item If $(JM(\cX)_{Y})_D\subseteq\overline{(m_YJM(\cX))_D}$ at the origin then $$JM(\cX)_Y\subseteq\overline{m_YJ_zM(\cX)}$$ at the origin, i.e, the W condition holds at the origin.
\item If $(JM(\cX)_{Y})_D\subseteq((m_YJM(\cX))_D)^{\dagger}$ at the origin then $$JM(\cX)_Y\subseteq(m_YJ_zM(\cX))^{\dagger}$$ at the origin.

\end{enumerate}
\end{lemma}

\begin{proof}
(1) Let $\phi: (\bC,0)\rightarrow(\cX,0)$ be an arbitrary analytic curve. By hypothesis and Corollary \ref{C13} (1) we have $$\phi^{*}(JM(\cX)_Y)\subseteq\phi^{*}(m_YJM(\cX)).$$ 

Thus, $\phi^{*}(JM(\cX)_{Y})\subseteq m_1\phi^{*}(JM(\cX)_Y)+\phi^{*}(m_YJ_zM(\cX))$. \\By Nakayama's Lemma we conclude that $$\phi^{*}(JM(\cX)_{Y})\subseteq \phi^{*}(m_YJ_zM(\cX)).$$

(2) The analogous proof goes through, working with the strict integral closure.
\end{proof}

The next result says the $iL_{m_Y}$ condition is independent of the projection onto $Y$.

\begin{proposition}\label{P9}
$(JM(\cX)_{Y})_D\subseteq\overline{((m_YJM(\cX))_D)}$ at the origin if and only if $(JM(\cX)_{Y})_D\subseteq\overline{(m_YJ_zM(\cX))_D}$ at the origin.
\end{proposition}

\begin{proof}
The implication $(\Longleftarrow)$ is obvious. Let us to prove $(\Longrightarrow)$.

Let $\phi=(\phi_1,\phi_2): (\bC,0)\rightarrow(\cX\underset{Y}{\times}\cX,(0,0))$ be an arbitrary analytic curve.

Let us prove that $$((z_i\circ\phi_1-z_i\circ\phi_2)(\frac{\partial F}{\partial y_{\ell}}\circ\phi_1),0)\in \phi^{*}((m_YJ_zM(\cX))_D),$$ for all $i\in\{1,...,n\}$ and $\ell\in\{1,...,k\}$. 

In fact, by Lemma \ref{L8} (1) we have $(JM(\cX))_Y\subseteq\overline{(m_YJ_zM(\cX))}$, so $\frac{\partial F}{\partial y_{\ell}}\circ\phi_1\in \phi_{1}^{*}(m_YJ_zM(\cX))$ and we can write $\frac{\partial F}{\partial y_{\ell}}\circ\phi_1=\sum\limits_{r,j}\beta_{rj}((z_r\frac{\partial F}{\partial z_j})\circ\phi_1)$, with $\beta_{rj}\in\cO_{\bC,0}$. Then, $((z_i\circ\phi_1-z_i\circ\phi_2)(\frac{\partial F}{\partial y_{\ell}}\circ\phi_1),0)=\sum\limits_{r,j}\beta_{rj}\phi^{*}(((z_i\circ\pi_1-z_i\circ\pi_2)((z_r\frac{\partial F}{\partial z_j})\circ\pi_1),0))\in \phi^{*}((m_YJ_zM(\cX))_D)$. 

Let us prove that $$\phi^{*}((m_YJM(\cX)_Y)_D)\subseteq m_1\phi^{*}((m_YJM(\cX))_D)+\phi^{*}((m_YJ_zM(\cX))_D).$$ 

In fact, it is enough to look to the images of the generators of $(m_YJM(\cX)_Y)_D$. For all $i,j\in\{1,...,n\}$ and $\ell\in\{1,...,k\}$ we have $$\phi^{*}((z_i\frac{\partial F}{\partial y_{\ell}})_D)=(z_i\circ\phi_2)(\frac{\partial F}{\partial y_{\ell}}\circ\phi_1,\frac{\partial F}{\partial y_{\ell}}\circ\phi_2)+((z_i\circ\phi_1-z_i\circ\phi_2)(\frac{\partial F}{\partial y_{\ell}}\circ\phi_1),0).$$
\noindent Hence, $\phi^{*}((z_i\frac{\partial F}{\partial y_{\ell}})_D)\in m_1\phi^{*}((m_YJM(\cX))_D)+\phi^{*}((m_YJ_zM(\cX))_D)$. By Nakayama's Lemma we conclude that $$\phi^{*}((m_YJM(\cX))_D)\subseteq \phi^{*}((m_YJ_zM(\cX))_D).$$ 
\end{proof}

While a similar result for $iL_A$ does not make sense, if we work with the strict $iL_A$ condition, then we get an analogous result.

The next result generalizes Proposition 4.2 of \cite{G1}. It extends 4.2 even in the hypersurface case, because we do not assume $\cX$ is a family of isolated singularities. In our setup, the integral closure condition that defines $iL_A$ makes sense at all points $(x,x')\in \cX\underset{Y}{\times}\cX$. In the proof of our main result \ref {T2.2.8} we will need  to know that the integral closure condition holds off $Y$. 

\begin{proposition}\label{P7}
	Consider the family $\cX$ as above, and assume condition $W$ holds for the pair $\cX_0, Y$ along $Y$.
	\begin{enumerate}
		\item[a)] If $(x,x)\in\Delta(\cX)-(\Sigma(\cX)\underset{Y}{\times}\Sigma(\cX))$ then $(JM(\cX)_{Y})_D\subseteq(J_zM(\cX))_D$ at $(x,x)$. 
		
		\item[b)] If $(x,x')\in \cX\underset{Y}{\times}\cX)$, $x\neq x'$,  Then $(JM(\cX)_{Y})_D\subseteq\overline{J_zM(\cX))_D}$ at $(x,x')$.
		
	\end{enumerate}
\end{proposition} 

\begin{proof}

Condition $W$ holds for the pair $\cX_0,Y$ at $(0,0)\in Y$ if and only if $JM(\cX)_Y\subseteq \overline{m_YJM_z(\cX)}$ (\cite{G3}). We can choose a neighborhood $\cU$ of $(0,0)$ such that this inclusion holds on the neighborhood. So, this implies that cosupports of $JM(\cX)$ and $JM_z(\cX)$ are the same on this neighborhood. Then $z\in\Sigma(X_y)$ if and only if $(z,y)\in\Sigma(\cX)$.

	(a) By hypothesis $x=(z,y)$ is a smooth point of $\cX$, hence $z$ is a smooth point of $X_y$.  Then the two modules $JM(\cX)$, $JM_z(\cX)$ agree at $(z,y)$, because they are both free of the same rank.  Hence, $(JM(\cX)_Y)_D\subseteq(JM_z(\cX))_D$ at $(x,x)$.

	(b) Assume $(x,x')\in\cU\underset{Y}{\times}\cU$, with $x\neq x'$.

	 then Proposition \ref{P2.13} implies, at $(x,x')$:
\begin{center}
	$(JM(\cX)_Y)_D=((JM(\cX)_Y)_x\circ\pi_1)\oplus((JM(\cX)_Y)_{x'}\circ\pi_2)$\\ $\subseteq(\overline{(JM_z(\cX))_{x}}\circ\pi_1)\oplus (\overline{(JM_z(\cX))_{x'}}\circ\pi_2)=\overline{(JM_z(\cX))_D}$	
\end{center}	
\end{proof}

The next result generalizes Theorem 4.3 of \cite{G1} and states that the infinitesimal Lipschitz condition A holds generically along the parameter space $Y$.

\index{Genericity of the $iL_A$ condition}
\begin{theorem}[Genericity Theorem]\label{T2.2.8}
	Consider the setup \ref{setup2.1}. Then there exists a dense Zariski open subset $U$ of $Y$ such that the infinitesimal Lipschitz condition A holds for the pair $(\cX-Y,U\cap Y)$ along $Y$.
\end{theorem}

\begin{proof}
	We can write a matrix of generators of $(JM_z(\cX))_D$ as $$[(JM_z(\cX))_D]=\left[\begin{matrix}
	JM_z(\cX)\circ\pi_1           &          0\\
	JM_z(\cX)\circ\pi_2          &  (0,(z_i\circ\pi_1-z_i\circ\pi_2)(\frac{\partial F}{\partial z_s}\circ\pi_2))_{i,s=1}^n
	\end{matrix}\right]$$ whose entries are in $\cO_{\cX\underset{Y}{\times} \cX}$. Since $(JM_z(\cX))_D$ is a sheaf of submodules of $\cO_{\cX\underset{Y}{\times} \cX}^{2p}$ then, choosing $S_1,...,S_{2p}$ as the homogeneous coordinates on $\bP^{2p-1}$, we can consider the sheaf of ideals of $\cO_{\cX\underset{Y}{\times} \cX\times\bP^{2p-1}}$ induced by $(JM_z(\cX))_D$, namely $\rho((JM_z(\cX))_D)$, which is generated by the entries of the vector $$[\begin{matrix}
	1  &  \frac{S_2}{S_1} & ... &\frac{S_{2p}}{S_1} 
	\end{matrix}]\cdot[(JM_z(\cX))_D]$$ on the chart $U_1:=\{[S_1,...,S_{2p}]\in\bP^{2p-1}\mid S_1\neq 0\}$ which is a dense Zariski open subset of $\bP^{2p-1}$.
	
	Denote by $N:=NB_{\rho((JM_z(\cX))_D)}(\cX\underset{Y}{\times}\cX\times\bP^{2p-1})$ the normalized blow-up of $\cX\underset{Y}{\times}\cX\times\bP^{2p-1}$ with respect to the sheaf of ideals $\rho((JM_z(\cX))_D)$ of $\cO_{\cX\underset{Y}{\times}\cX\times\bP^{2p-1}}$. Consider the projection map $\pi: N\rightarrow \cX\underset{Y}{\times}\cX\times\bP^{2p-1}$ and let $E\subseteq N$ be the normalized exceptional divisor. To prove this theorem we use the module criterion (see Proposition 3.5 in \cite{GK}), i.e, in order to verify the condition $(JM(\cX)_Y)_D\subseteq\overline{(JM_z(\cX))_D}$ in a dense Zariski open subset $U$ of $Y$, it suffices to check that on each component of the exceptional divisor, the pullback of the element $\rho((\frac{\partial F}{\partial y})_D)$ to the normalized blow-up is in the pullback of $\rho((JM_z(\cX))_D)$, for every coordinate $y$ in the parameter space.
	
	Let $\mathbf{p}: \cX\subseteq\bC^n\times Y\rightarrow Y$ be the projection onto $Y$. For each $\ell\in\{1,2\}$ consider the projection map $p_{\ell}: \cX\underset{Y}{\times}\cX\times\bP^{2p-1}\rightarrow \cX$ on the $\ell^{\mbox{th}}$ factor and $\bar{\pi_{\ell}}:N\rightarrow \cX$ given $\bar{\pi_{\ell}}:=p_{\ell}\circ\pi$.
	
	\begin{center}
		\begin{tikzpicture}
		\matrix (m) [matrix of math nodes,row sep=3em,column sep=3em,minimum width=2em]
		{
			&	& N &  & \\
			Y	&	\cX  &  \cX\underset{Y}{\times}\cX\times\bP^{2p-1} &   \cX & Y \\
			
			&	& \cX\underset{Y}{\times}\cX  &  & \\};
		\path[-stealth]
		(m-1-3) edge node [right] {$\pi$} (m-2-3)
		(m-2-3) edge node [right] {$\bar{p}$} (m-3-3)
		(m-1-3) edge node [above] {$\bar{\pi_1}$} (m-2-2)
		(m-1-3) edge node [right] {$\bar{\pi_2}$} (m-2-4)
		(m-2-3) edge node [above] {$p_1$} (m-2-2)
		(m-2-3) edge node [above] {$p_2$} (m-2-4)
		(m-3-3) edge node [left] {$\pi_1$} (m-2-2)
		(m-3-3) edge node [right] {$\pi_2$} (m-2-4)
		(m-2-2) edge node [above] {$\mathbf{p}$} (m-2-1)
		(m-2-4) edge node [above] {$\mathbf{p}$} (m-2-5);
		\end{tikzpicture}
	\end{center}
	
	By Proposition \ref{P7} we need only consider those components of the exceptional divisor which project to $Y$ under the map to $\cX\underset{Y}{\times}\cX$. Since we are working over a dense Zariski open subset of $Y$ we may assume that every such component maps surjectively onto $Y$. Since $N$ is a normal space and $E$ has codimension $1$ in $N$ then we can work at a point $q$ of the normalized exceptional divisor $E$ such that $E$ is smooth at $q$, $N$ is smooth at $q$ and the projection to $Y$ is a submersion at $q$. Thus we can choose coordinates $(y',u',x')$ such that $y'_i=y_i\circ\mathbf{p}$, $i\in\{1,...,k\}$, $u'$ defines $E$ locally with reduced structure and $\frac{\partial u'}{\partial y'_i}=0$, $i\in\{1,...,k\}$, i.e, $u'$ and $y'$ are independent coordinates. Working on the subset $U_1\subseteq\bP^{2p-1}$, since $\cX$ is defined by $F$ then the germ of $$[\begin{matrix}
	1 & \frac{S_2}{S_1} & ... & \frac{S_{2p}}{S_1}
	\end{matrix}]\cdot \left[\begin{matrix}
	F_1\circ p_1 \\
	\vdots \\
	F_p\circ p_1 \\
	F_1\circ p_2 \\
	\vdots \\
	F_p\circ p_2
	\end{matrix}\right]=0$$ is identically zero on $\cX\underset{Y}{\times}\cX\times\bP^{2p-1}$. Pull this back to $N$ by $\pi$ and take the partial derivative with respect to $y'$ at $q$. We get by the chain rule:
	
	$$[\begin{matrix}
	1 & \frac{S_2}{S_1} & ... & \frac{S_{2p}}{S_1}
	\end{matrix}]\cdot \left[\begin{matrix}
	\frac{\partial F_1}{\partial y}\circ\bar{\pi_1}+\sum\limits_{i=1}^{n}\left(\frac{\partial F_1}{\partial z_i}\circ\bar{\pi_1}\right)\left(\frac{\partial(z_i\circ\bar{\pi_1})}{\partial y'}\right) \\
	\vdots \\
	\frac{\partial F_p}{\partial y}\circ\bar{\pi_1}+\sum\limits_{i=1}^{n}\left(\frac{\partial F_p}{\partial z_i}\circ\bar{\pi_1}\right)\left(\frac{\partial(z_i\circ\bar{\pi_1})}{\partial y'}\right) \\
	\frac{\partial F_1}{\partial y}\circ\bar{\pi_2}+\sum\limits_{i=1}^{n}\left(\frac{\partial F_1}{\partial z_i}\circ\bar{\pi_2}\right)\left(\frac{\partial(z_i\circ\bar{\pi_2})}{\partial y'}\right) \\
	\vdots \\
	\frac{\partial F_p}{\partial y}\circ\bar{\pi_2}+\sum\limits_{i=1}^{n}\left(\frac{\partial F_p}{\partial z_i}\circ\bar{\pi_2}\right)\left(\frac{\partial(z_i\circ\bar{\pi_2})}{\partial y'}\right) \\
	\end{matrix}\right]=0 \eqno{(\star)}$$
	
	Since $F_j\circ\bar{\pi_1}=F_j\circ\bar{\pi_2}=0$ for all $j\in\{1,...,p\}$ then there is no term involving the derivatives of the homogeneous coordinates with respect to $y'$. Notice that all $z_i$ vanish along $Y$, so $z_i\circ\bar{\pi_1}$ and $z_i\circ\bar{\pi_2}$ vanish along $E$ at $q$, then we can assume that the order of vanishing of $z_1\circ\bar{\pi_{\ell}}$ is minimal among $\{z_i\circ\bar{\pi_{\ell}}\}$ and that the strict transform of $z_1\circ\bar{\pi_{\ell}}$ does not pass through $q$, $\forall \ell\in\{1,2\}$.
	
	By equation $(\star)$ we have that $\rho((\frac{\partial F}{\partial y})_D)\circ\pi=[\begin{matrix}
	1 & \frac{S_2}{S_1} & ... & \frac{S_{2p}}{S_1}
	\end{matrix}]\cdot\left[\begin{matrix}
	\frac{\partial F_1}{\partial y}\circ\bar{\pi_1}\\
	\vdots\\
	\frac{\partial F_p}{\partial y}\circ\bar{\pi_1}\\
	\frac{\partial F_1}{\partial y}\circ\bar{\pi_2}\\
	\vdots \\
	\frac{\partial F_p}{\partial y}\circ\bar{\pi_2}
	\end{matrix}\right]=-v$, where $$v:=[\begin{matrix}
	1 & \frac{S_2}{S_1} & ... & \frac{S_{2p}}{S_1}
	\end{matrix}]\cdot\left[\begin{matrix}
	\sum\limits_{i=1}^{n}\left(\frac{\partial F_1}{\partial z_i}\circ\bar{\pi_1}\right)\left(\frac{\partial(z_i\circ\bar{\pi_1})}{\partial y'}\right) \\
	\vdots \\
	\sum\limits_{i=1}^{n}\left(\frac{\partial F_p}{\partial z_i}\circ\bar{\pi_1}\right)\left(\frac{\partial(z_i\circ\bar{\pi_1})}{\partial y'}\right) \\
	\sum\limits_{i=1}^{n}\left(\frac{\partial F_1}{\partial z_i}\circ\bar{\pi_2}\right)\left(\frac{\partial(z_i\circ\bar{\pi_2})}{\partial y'}\right) \\
	\vdots \\
	\sum\limits_{i=1}^{n}\left(\frac{\partial F_p}{\partial z_i}\circ\bar{\pi_2}\right)\left(\frac{\partial(z_i\circ\bar{\pi_2})}{\partial y'}\right) \\
	\end{matrix}\right].$$ 
	
	In order to simplify the notation, for each $i\in\{1,...n\}$ define 
	
	$$w_i:=[\begin{matrix}
	1 & \frac{S_2}{S_1} & ... & \frac{S_{2p}}{S_1}
	\end{matrix}]\cdot\left[\begin{matrix}
	\left(\frac{\partial F_1}{\partial z_i}\circ\bar{\pi_1}\right)\left(\frac{\partial(z_i\circ\bar{\pi_1})}{\partial y'}\right) \\
	\vdots \\
	\left(\frac{\partial F_p}{\partial z_i}\circ\bar{\pi_1}\right)\left(\frac{\partial(z_i\circ\bar{\pi_1})}{\partial y'}\right) \\
	\left(\frac{\partial F_1}{\partial z_i}\circ\bar{\pi_2}\right)\left(\frac{\partial(z_i\circ\bar{\pi_1})}{\partial y'}\right) \\
	\vdots \\
	\left(\frac{\partial F_p}{\partial z_i}\circ\bar{\pi_2}\right)\left(\frac{\partial(z_i\circ\bar{\pi_1})}{\partial y'}\right) \\
	\end{matrix}\right]$$ and 
	
	$$\tilde{w}_i:=-[\begin{matrix}
	1 & \frac{S_2}{S_1} & ... & \frac{S_{2p}}{S_1}
	\end{matrix}]\cdot\left[\begin{matrix}
	0 \\
	\vdots \\
	0 \\
	\left(\frac{\partial F_1}{\partial z_i}\circ\bar{\pi_2}\right)\left(\frac{\partial(z_i\circ\bar{\pi_1})}{\partial y'}-\frac{\partial(z_i\circ\bar{\pi_2})}{\partial y'}\right) \\
	\vdots \\
	\left(\frac{\partial F_p}{\partial z_i}\circ\bar{\pi_2}\right)\left(\frac{\partial(z_i\circ\bar{\pi_1})}{\partial y'}-\frac{\partial(z_i\circ\bar{\pi_2})}{\partial y'}\right) \\
	\end{matrix}\right].$$
	
	Clearly $v=\sum\limits_{i=1}^{n}(w_i+\tilde{w}_i)$. For every $i\in\{1,...,n\}$ we have that \\$w_i=\frac{\partial(z_i\circ\bar{\pi_1})}{\partial y'}\pi^*(\rho((\frac{\partial F}{\partial z_i})_D))\in\pi^*(\rho((JM_z(\cX))_D))$. 
	
	Now it suffices to check that $\tilde{w}_i\in\pi^*(\rho((JM_z(\cX))_D))$, $\forall i\in\{1,...,n\}$. Since the pullback of the ideal $\rho((JM_z(\cX))_D)$ is locally principal then we can work at a point $q$ such that $\pi^*(\rho((JM_z(\cX))_D))$ is generated by $u'^{r}$, a power of $u'$. Since $\cO_{N,q}$ is a normal ring then Lemma 1.12 of \cite{LT} implies that the ideal $\pi^*(\rho((JM_z(\cX))_D))$ is integrally closed, i.e, $\overline{\pi^*(\rho((JM_z(\cX))_D))}=\pi^*(\rho((JM_z(\cX))_D))$. So, it is enough to prove that $\tilde{w}_i\in\overline{\pi^*(\rho((JM_z(\cX))_D))}$, for all $i\in\{1,...,n\}$. Let $i\in\{1,...,n\}$ be arbitrary. We use the curve criterion. Let $\tilde{\phi}:(\bC,0)\rightarrow(N,q)$ be an analytic curve. We can choose $\tilde{\phi}$ such that $\phi:(\bC,0)\rightarrow(\cX\underset{Y}{\times}\cX\times\bP^{2p-1},\pi(q))$ given by $\phi:=\pi\circ\tilde{\phi}$ meets the dense Zariski open subset $U_1$, $\phi=(\phi_1,\phi_2,\psi)$ and $\psi=\left[1,\frac{\psi_2}{\psi_1},...,\frac{\psi_{2p}}{\psi_1}\right]$. Further, $\tilde{\phi}$ can be chosen such that $\tilde{\phi}$ is transverse to the component so that $u'\circ\tilde{\phi}=t$, where $t$ is the generator of the maximal ideal of $\cO_{\bC,0}$. Hence, the pullback of the ideal $\pi^*(\rho((JM_z(\cX))_D))$ is generated by $t^r$. Consider the element

	$$\hat{w}_i:=-[\begin{matrix}
	1 & \frac{S_2}{S_1} & ... & \frac{S_{2p}}{S_1}
	\end{matrix}]\cdot\left[\begin{matrix}
	0 \\
	\vdots \\
	0 \\
	\left(\frac{\partial F_1}{\partial z_i}\circ\bar{\pi_2}\right)(z_i\circ\bar{\pi_1}-z_i\circ\bar{\pi_2}) \\
	\vdots \\
	\left(\frac{\partial F_p}{\partial z_i}\circ\bar{\pi_2}\right)(z_i\circ\bar{\pi_1}-z_i\circ\bar{\pi_2}) \\
	\end{matrix}\right].$$
	
	\noindent Notice that $$\hat{w}_i=-\pi^*(\rho((0,(z_i\circ\pi_1-z_i\circ\pi_2)(\frac{\partial F}{\partial z_i}\circ\pi_2))_{\in(JM_z(\cX)_D)}))\in\pi^*(\rho((JM_z(\cX))_D)).$$ Since $y'$ and $u'$ are independent coordinates then the order of $\frac{\partial(z_i\circ\bar{\pi_1})}{\partial y'}-\frac{\partial(z_i\circ\bar{\pi_2})}{\partial y'}$ in $u'$ is the same as the order of $z_i\circ\bar{\pi_1}-z_i\circ\bar{\pi_2}$ in $u'$. Then the pullback of both have the same order in $t$, so there exists an invertible element $\alpha_i\in\cO_{\bC,0}$ such that $$\tilde{\phi}^*\left(\frac{\partial(z_i\circ\bar{\pi_1})}{\partial y'}-\frac{\partial(z_i\circ\bar{\pi_2})}{\partial y'}\right)=\alpha_i(\tilde{\phi}^*(z_i\circ\bar{\pi_1}-z_i\circ\bar{\pi_2})).$$
	Hence, $\tilde{\phi}^*(\tilde{w}_i)=-\sum\limits_{j=1}^{p}\left(\frac{\partial F_j}{\partial z_i}\circ\phi_2\right)\left(\tilde{\phi}^*\left(\frac{\partial(z_i\circ\bar{\pi_1})}{\partial y'}-\frac{\partial(z_i\circ\bar{\pi_2})}{\partial y'}\right)\right)\frac{\psi_{p+j}}{\psi_1}\\=\alpha_i\tilde{\phi}^*(\hat{w}_i)\in\tilde{\phi}^*(\pi^*(\rho((JM_z(\cX))_D)))$.
	
	Therefore, $\tilde{w}_i\in\overline{\pi^*(\rho((JM_z(\cX))_D))}$, for all $i\in\{1,...,n\}$.	
\end{proof}

In general, we do not have an answer about the genericity of the $iL_{m_Y}$ condition. However, in the case that $\cX$ is a $1$-parameter family defined by a map $F$ which has all components weighted homogeneous polynomials of the same type, then it is easy to conclude that $JM(\cX)_Y\subseteq m_YJM_z(\cX)$ at any point $x=(z,y)\in \cX$, with $y\neq 0$. In particular, $iL_{m_Y}$ is generic.

\section{The genericity theorem applied in a family of hyperplane sections}

Given $X$ an analytic variety with isolated singularity at the origin, we can consider the sections of $X$ by hyperplanes. One natural question is if there exists a generic set of hyperplanes for which the family of hyperplanes sections satisfies the infinitesimal Lipschitz condition A. We show this is true. First, we recall some important notions in order to make precise statements. Fore more details see \cite{G6}.

Let us work on the Grassmanian modification of $X=f^{-1}(0)$, defined by an analytic map $f:(\bC^n,0)\rightarrow(\bC^p,0)$, $X$ with isolated singularity at the origin, $n\geq p$. 

For each $y=[y_1,...,y_n]\in\bP^{n-1}$, consider the hyperplane on $\bC^n$ given by $$H_y:=\{z=(z_1,...,z_n)\in\bC^n\mid z\cdot y:=\sum\limits_{i=1}^{n}z_iy_i=0\}.$$ Let $E_{n-1}$ be the canonical bundle over $\bP^{n-1}$, i.e, $$E_{n-1}:=\{(z,y)\in\bC^n\times\bP^{n-1}\mid z\in H_y\}.$$ Consider the projection map $\beta:E_{n-1}\rightarrow\bC^n$. We call $\tilde{X}:=\beta^{-1}(X)$ the $\mathbf{(n-1)}$\textbf{-Grassmanian modification of} $\mathbf{\mathbf{X}}$. Here we simply refer to the $(n-1)$-modification as the Grassmanian modification of $X$. We can see $\bP^{n-1}$ embedded into $E_{n-1}$ as the zero section of the bundle $E_{n-1}$, which allows us to think of $0\times\bP^{n-1}$ as a stratum of $\tilde{X}$. Note that the projection to $0\times\bP^{n-1}$ makes $\tilde{X}$ a family of analytic sets with $0\times\bP^{n-1}$ as the parameter space, which we denote by $Y$. The members of this family are just $\{H_y\cap X\}$ as $y$ varies in $\bP^{n-1}$.

Consider the chart $U_n:=\{[y_1,...,y_n]\in\bP^{n-1}\mid y_n\neq 0\}\\=\{[y_1,...,y_{n-1},-1]\mid (y_1,...,y_{n-1})\in\bC^{n-1}\}\equiv\bC^{n-1}$ which is a dense Zariski open subset of $\bP^{n-1}$. Working on the dense Zariski open subset $E_{n-1}\cap(\bC^n\times U_n)$ of $E_{n-1}$, we have local coordinates given by $(z_1,...,z_n,y_1,...,y_{n-1})$. In these coordinates, the projection map $\beta$ satisfies the equation \\ $\beta(z_1,...,z_n,y_1,...,y_{n-1})=(z_1,...,z_{n-1},\sum\limits_{i=1}^{n-1}y_iz_i)$.

Consider the analytic map $F:=f\circ\beta: E_{n-1}\cap(\bC^n\times U_n)  \rightarrow  \bC^p$. Thus, $F^{-1}(0)=\beta^{-1}(f^{-1}(0))=\beta^{-1}(X)=\tilde{X}$, hence $\tilde{X}$ is defined by $F$. For each $y=(y_1,...,y_{n-1})\\ \equiv[y_1,...,y_{n-1},-1]\in U_n$, let $F_y:\bC^n\rightarrow\bC^p$ given by $F_y(z):=F(z,y)$ and let $\tilde{X}_y:=F^{-1}(0)$. In these coordinates, clearly $\tilde{X}_y=(f^{-1}(0))\cap H_y=X\cap H_y$. Therefore, $F$ defines the family of sections of $X$ by the hyperplanes $H_y$, as $y$ varies on the dense Zariski open subset $U_n$ of $\bP^{n-1}$.

The next result generalizes Theorem 4.4 of \cite{G1}.

\begin{theorem}
 There exists a non-empty Zariski open subset $U$ of $\bP^{n-1}$, such that the $iL_A$ condition holds for the pair $(\tilde{X}-U,U)$ along $U$.
\end{theorem}

\begin{proof}
	As we have seen, $\tilde{X}$ is a family defined by the above analytic map $F$. 
	
	Let us prove that $\tilde{X}_y$ has isolated singularity at $(0,y)$ for all $y$ varying in a non-empty Zariski open subset $U'$ of $U_n$. In fact, we already know that the set of limiting tangent hyperplanes of $X$ at the origin is a Zariski proper closed subset of $\bP^{n-1}$. Call this set $W$. Let \\$U':=U_n-(W\cap U_n)$. Since $\bP^{n-1}$ is irreducible then $U_n$ is irreducible. Since $U_n$ is a dense subset of $\bP^{n-1}$ then $W\cap U_n$ also is a proper Zariski closed subset of $U_n$, hence $U'$ is a dense Zariski open subset of $U_n$. Let $y\in U'$. We want to show that $(0,y)$ is an isolated singularity of $\tilde{X}_y$. By hypothesis, $H_y$ is not a limiting tangent hyperplane of $X$ at the origin, and by Lemma 4.1 (a) of \cite{GK} we have that $\overline{JM(X)_{H_y}}=\overline{JM(X)}$ at the origin, where $JM(X)_{H_y}:=\{\frac{\partial f}{\partial v}\mid v\in H_y\}$. Thus, in a neighborhood of the origin, the generic rank of $JM(X)$ and $JM(X)_{H_y}$ is the same. Thus, if we take $z$ in this neighborhood, such that $z\in H_y$, $z\neq0$ then the generic rank of $JM(\tilde{X}_y)=JM(X\cap H_y)$ at $z$ is the generic rank of $JM(X)_{H_y}$ at $z$, which is the generic rank of $JM(X)$ at $z$. Since $z\neq 0$ and $X$ has isolated singularity at the origin then we can choose this neighborhood so that $z$ is a non-singular point of $X$, which implies that $z$ is not a singular point of $\tilde{X}_y$. Therefore, $\tilde{X}_y$ has isolated singularity at the origin, for all $y\in U'$.
	
	Now, the existence of $U$ follows from Theorem \ref{T2.2.8}.
\end{proof}

Let us go back to the discussion before the last theorem. We have seen that $\tilde{X}$ is defined by the map $F:E_{n-1}\cap(\bC^n\times U_n)\rightarrow\bC^p$ given by $F(z,y)=f\circ\beta(z,y)$. From the chain rule we have $\frac{\partial F}{\partial y_i}=z_i\left(\frac{\partial f}{\partial z_n}\circ\beta\right)$ and $\frac{\partial F}{\partial z_i}=\frac{\partial f}{\partial z_i}\circ\beta+\sum\limits_{j=1}^{n-1}y_j\left(\frac{\partial f}{\partial z_n}\circ\beta\right)$, for all $i\in\{1,...,n-1\}$, and $\frac{\partial F}{\partial z_n}=0$. Thus, we have immediately the next result, which is a generalization of Corollary 4.5 of \cite{G1}.

\begin{corollary}
	The point $(0,P)\in E_{n-1}\cap(\bC^n\times U_n)$ belongs to the Zariski open subset of the last theorem if and only if $\left(z_i(\frac{\partial f}{\partial z_n}\circ\beta)\right)_D\in\overline{(JM_z(\tilde{X}))_D}$ at $(0,P)$, for all $i\in\{1,...,n-1\}$.
\end{corollary}

In \cite{G1} Gaffney gave a description of these generic hyperplanes using analytic invariants in the jacobian ideal case. Now we generalize this description for the jacobian module case. For the rest of this section we assume that the hyperplanes $H_y$ are not limiting tangent hyperplanes of $(X,0)$. As we have seen, this implies that $\tilde{X}_y=X\cap H_y$ has isolated singularity at the origin and $\overline{JM(X)_{H_y}}=\overline{JM(X)}$ at the origin.

The invariants that we use here appeared earlier in the previous section. Since $\tilde{X}_y$ has isolated singularity at the origin then by Corollary \ref{C2.18} the multiplicity of the pair $e((JM(\tilde{X}_y))_D,(\overline{JM(\tilde{X}_y)})_D)$ is well defined. 

The proof that the minimal value of $e((JM(\tilde{X}_y))_D,(\overline{JM(\tilde{X}_y)})_D)$ identifies generic hyperplanes will be done using the Multiplicity Polar Theorem (see Corollary 1.4 \cite{G7}). Now we identify the modules we will use.

We work on the fibered product $\tilde{X}\underset{\bP^{n-1}}{\times}\tilde{X}\subseteq X\times\bP^{n-1}\times X$. Let $N:=(\beta^*(\overline{JM(X)}))_D$ and $M:=(JM_z(\tilde{X}))_D$, considering $\tilde{X}$ defined by the analytic map $F: E_{n-1}\cap(\bC^n\times U_n)\rightarrow\bC^p$, given by $F(z,y)=f\circ\beta(z,y)$. Clearly $M$ restricted to the fiber of the family $\tilde{X}$ over the hyperplane $H_y$ is just $(JM(X\cap H_y))_D$ and $N$ restricted to $H_y$ is $(\overline{JM(X)}\mid_{H_y})_D$. Further, since we are assuming that $H_y$ is not a limiting tangent hyperplane of $(X,0)$ then $\overline{JM(X)}\mid_{H_y}=\overline{JM(X)_{H_y}}$, hence $N$ restricted to $H_y$ is $(\overline{JM(\tilde{X}_y)})_D$. Therefore, the multiplicity of the pair $(M\mid_{H_y},N\mid_{H_y})$ is the same as $e((JM(\tilde{X}_y))_D,(\overline{JM(\tilde{X}_y)})_D)$.

The next result is a generalization of Theorem 4.6 of \cite{G1}.

\begin{theorem}
	Under the above notations, let $U$ be the set of hyperplanes which are not limiting tangent hyperplanes of $(X,0)$. Shrinking $U$ if necessary, one has: 
	
	\begin{enumerate}
		\item [a)] The map $$\begin{matrix}
		U & \rightarrow & \bZ \\
		H_y & \mapsto  & e((JM(\tilde{X}_y))_D,(\overline{JM(\tilde{X}_y)})_D)
		\end{matrix}$$
		
		\noindent is upper semicontinuous on $U$;
		
		\item[b)] The $iL_A$ condition holds along $U$ at a hyperplane $H_y$ for which the value of \\$e((JM(\tilde{X}_y))_D,(\overline{JM(\tilde{X}_y)})_D)$ is minimal.
	\end{enumerate}
\end{theorem} 

\begin{proof}
	(a) By the definition of $U$, $(\overline{JM(\tilde{X}_y)})_D$ is the restriction of $N$ to the fiber $\tilde{X}_y=X\cap H_y$. Since $N$ has no polar variety with the same codimension of $U$ then the multiplicity polar theorem implies that $e((JM(\tilde{X}_y))_D,(\overline{JM(\tilde{X}_y)})_D)$ is upper semicontinuous (this is avoidable by shrinking the $Z$-open subset $U$).
	
	(b) Suppose $H_y\in U$ gives the minimal value of the multiplicity. Since this value already is minimal then it cannot go down, hence it must be constant. This implies that the polar variety of $M$ of the same codimension as $U$ is empty, which puts restrictions on the size of the fiber of $\textrm{Proj}(\cR(M))$. We already know that $\{\frac{\partial F}{\partial y_i}\}_{i=1}^{n-1}$ are in $\overline{M}$ generically. Since the dimension of the fiber of $\textrm{Proj}(\cR(M))$ is bounded, then by Theorem A1 of \cite{KT} we have that $\{\frac{\partial F}{\partial y_i}\}_{i=1}^{n-1}$ are in $\overline{M}$ at $H_y$.
\end{proof}

\begin{remark}
It may be necessary to shrink $U$ to avoid points in the parameter space where the polar variety of $N$ of dimension $n-1$ contains the point.
\end{remark}

\section{Bi-Lipschitz equisingularity in an ICIS family of irreducible curves}\label{sec33}

In this section we prove that a strengthened version of the $iL_A$ condition implies bi-Lipschitz equisingularity if $\cX$ is a family of irreducible ICIS curves.

First, the next result gives us conditions involving the double and the integral closure of modules so that we can construct Lipschitz vector fields. In what follows, let $\tilde\cX$ denote the normalization of $\cX$.

\begin{proposition}\label{P2.24}
	Suppose $\cX\subseteq\bC^n\times \bC^k$, with $\dim \cX=k+1$. Let $M$ be an $\cO_{\cX}$-submodule of $\cO_{\cX}^k$ of  rank $k$ off $Y:=0\times \bC^k$, with $r$ generators. Let $M_k$ be a reduction of $M$ generated by the first $k$ columns of $[M]$. Let $h\in\cO_{\cX}^k$. Let $c(h)$ be the meromorphic vector field defined by the solution of the equation $$[M_k]\cdot c(h)=h$$ off $Y$ (using Cramer's Rule). If $h\in\overline{m_YM}$ at $x$ and $h_D\in\overline{(M_k)_{D,Y}}$ at $(x,x)$ then the vector field $c(h)$ is Lipschitz rel $Y$, i.e, $c(h)-c(h)'\in\overline{I_{\Delta}\cO_{\tilde\cX\times_Y \tilde\cX}^k}$ at $(x,x)$.
\end{proposition}

\begin{proof}
	Let us use the curve criterion. Let $\phi:\bC,0\rightarrow \cX\times_Y \cX,(y,x,x)$ be a curve, with coordinates $ \phi_1(t),\phi_2(t)$, $\pi_Y\circ\phi_1(t)=\pi_Y\circ  \phi_2(t)$.
	
	First suppose $\phi_1(t)\equiv (0, \pi_Y\circ\phi_1(t)) $. In this case we have $$\left[
	\begin{matrix}
	h\circ\phi_1 \\
	h\circ \phi_2
	\end{matrix}
	\right]
	=\left[\begin{matrix}
	0 \\
	h\circ \phi_2
	\end{matrix}\right]
	=\left[\begin{matrix}
	0 \\
	[M_k]'\cdot c(h)
	\end{matrix}\right]
	$$
	
	\noindent off $0\times Y$.  Then, because $h\circ\phi_2\in(\phi_2)^*(m_YM_k')$ implies that $h\circ\phi_2=M'_k(v)$, $v\in\phi^*_2 m_Y\cO_{\bC,0}^k$,  we have $M'_k(c(h)'-v)=0$ and $c(h)'=v\in (z_i')\cO_{\cX}^k$ along $\phi_2$. (In particular, this shows $c(h)$ is well defined along curves in $\cX$, hence a smooth function on $\tilde\cX$.) Since $c(h)=0$ along $\phi_1$ then $\phi^*(c(h)-c(h)')\in\phi^*(I_{\Delta}\cO_{\tilde\cX\times \tilde\cX}^k)$.
	
	The case where $\phi_2\equiv (0, \pi_Y\circ\phi_2(t)) $ is analogous.
	
	Now assume that $\pi_{\cX}\circ\phi_1,\pi_{\cX}\circ\phi_2\neq 0$. Since $(h_D)\in\overline{(M_k)_{D,Y}}$ then $$(0,[M_k]\circ \phi_2.(c(h)\circ \phi_2-c(h)\circ \phi_1)\in\Phi^*{(M_k)_D} .$$
	
	So, $\begin{bmatrix}
	0\\
	M_k'(c(h)'-c(h))
	\end{bmatrix}\equiv\begin{bmatrix}
	M_k\circ\phi_1\\
	M_k\circ\phi_2
	\end{bmatrix}(v)\mod (0,\phi_2^*(M'_k I_{\Delta}\cO_{X\times X}^k))$.
	
	\noindent Since $M_k\circ\phi_1$ has rank $k$ generically, then $v$ has to be zero. Hence $M_k'(c(h)'-c(h))\in M_k'I_{\Delta}\cO_{X\times X}^k$ at $(x,x)$ along $\phi$.
	
	Therefore $c(h)'-c(h)\in \phi^*(I_{\Delta}\cO_{X\times X}^k)$ along $\phi$ at $(x,x)$. Then these three cases show that $c(h)'-c(h)\in\overline{I_{\Delta}\cO_{\tilde\cX\times_Y \tilde\cX}^k}$ at $(x,x)$.
\end{proof}

Let $\cX\subseteq\bC\times\bC^n$ be a  family of irreducible, ICIS curves.  Suppose that the family $\cX$ is Whitney equisingular, with the the parameter axis as the singular set of the family.  This implies that the multiplicity, Milnor number and $\delta$ invariant of each curve in the family is independent of parameter (\cite {G8}).  In turn, this implies that we have a normalization $F:\bC\times\bC\rightarrow \cX\subseteq\bC\times\bC^n$ which is a homeomorphism (cf. cor 1 p 605 \cite{T76}).  After a coordinate change in source and target we can put our family into a nice form. 

\begin{proposition} \label{normal form} Suppose the family $\cX$ is Whitney equisingular, with the parameter axis as the singular set of the family. Suppose the normalization of the family is $F:\bC\times\bC\rightarrow \cX\subseteq\bC\times\bC^n$. Then after holomorphic coordinate changes in $\bC\times \bC$ and $\bC\times\bC^n$, $F$ has normal form

$$F(t,s)=(t,F_1(t,s),...,F_{n-1}(t,s),s^p)$$ where for each parameter $t$ the order of $$s\mapsto F_i(t,s)$$ is greater than $p$, $\forall i\in\{1,...,n-1\}$.\end{proposition}

\begin{proof} Since $F$ is a homeomorphism with singular locus the $t$ axis and constant multiplicity it follows that we can write $F$ as 
$$F(s,t)=(t,s^p(\bv(t)))mod (s^{p+1}), \bv(0)\ne 0.$$ After a scale change and possible permutation of the coordinates in the target we can assume $\bv(0)=(v_1,\dots, 1)$.  Let $L(t,z_1\dots,z_n)=(t, z_1+v_1(t)z_n,\dots, z_{n-1}+z_nv_{n-1}(t), v_n(t)z_n)$. Then $L$ is biholomorphic and 
$$L\circ(t,0,\dots,s^p)=(t, s^p(\bv(t)).$$

Then $L^{-1}\circ F(s,t)=(t,0,\dots,0, s^p) mod (s^{p+1})$.

Let $G$ denote $L^{-1}\circ F(s,t)$. The last coordinate of $G$ has the form $s^p(c(s,t))$, $c(0,0)=1$. Let $R(s,t)=(t, sc^{1/p}(s,t))$. Then $R$ is biholomorphic and $(t,s^p)\circ R=(t, s^p(c(s,t)))$. This implies that $L^{-1}\circ F(s,t)\circ R^{-1}$ has the desired form. \end {proof}

\begin{remark} If $F$ is a parameterization for a family of irreducible ICIS curves with the form of \ref{normal form}, then it is clear that the curves in the family have the same multiplicity and $\delta$ invariant, hence the same Milnor number, hence form a Whitney equisingular family. \end{remark}
Let $\cX\subseteq\bC\times\bC^n$ be an ICIS family of irreducible curves. Assume that we have a normalization $F:\bC\times\bC\rightarrow \cX\subseteq\bC\times\bC^n$ which is a homeomorphism, and suppose that the family $\cX$ is Whitney equisingular. Let $p$ be the multiplicity of $\cX$ and assume we can write $$F(t,s)=(t,F_1(t,s),...,F_{n-1}(t,s),s^p)$$ where for each parameter $t$ the order of $$s\mapsto F_i(t,s)$$ is greater than $p$, $\forall i\in\{1,...,n-1\}$.

Notice that if $\frac{\partial F_1}{\partial t},...,\frac{\partial F_n}{\partial t}$ are Lipschitz functions on $\cX$ rel $Y$, then one can build a canonical vector field defined on $\cX$, which is Lipschitz on fibers. Thus, the flow of this vector field makes $\cX$ a bi-Lipschitz equisingular family of curves.

Let $G:\bC\times\bC^n\rightarrow\bC^q$ be an analytic map that defines $\cX$, i.e, $\cX=G^{-1}(0)$. Consider the jacobian module $JM(\cX)=JM(G)$. Here we consider the double relative to the parameter space $Y=\bC\times 0\equiv \bC$.

Thanks to the assumed normal form, as we shall see, $JM_z(G)$ has a minimal reduction generated by the first $n-1$ partial derivatives which we denote by $DG_{n-1}$.  We strengthen the $iL_A$ condition, by asking that $\left(\frac{\partial G}{\partial t}\right)_D\in\overline{(DG_{n-1})_D}$.

The following theorem gives us an infinitesimal condition for bi-Lipschitz equisingularity.

\begin{proposition}
	With the above notations, the functions $\frac{\partial F_1}{\partial t},...,\frac{\partial F_n}{\partial t}$ are Lipschitz rel $Y$ if and only if, $$\left(\frac{\partial G}{\partial t}\right)_D\in\overline{(DG_{n-1})_{D,Y}}.$$ . 
	
	In particular, if $(JM(\cX)_Y)_{D,Y}\subseteq\overline{(DG_{n-1})_{D,Y}}$ then $\cX$ is bi-Lipschitz equisingular.
\end{proposition}

\begin{proof}
	Clearly\\ $0=\left[\frac{\partial (G\circ F)}{\partial t}\right] =\left[\begin{matrix}
	\frac{\partial G}{\partial t}\circ F & \frac{\partial G}{\partial z_1}\circ F  &...& \frac{\partial G}{\partial z_{n-1}}\circ F & \frac{\partial G}{\partial z_n}\circ F
	\end{matrix}\right]\cdot\left[\begin{matrix}
	1 \\
	\frac{\partial F_1}{\partial t} \\
	\vdots \\
	\frac{\partial F_{n-1}}{\partial t} \\
	0
	\end{matrix}\right]$
	
	\noindent which implies that $\frac{\partial G}{\partial t}\circ F =-[DG_{n-1}\circ F]\cdot \left[\begin{matrix}
	\frac{\partial F_1}{\partial t} \\
	\vdots \\
	\frac{\partial F_{n-1}}{\partial t}
	\end{matrix}\right]$. 
	
	We also have that $0=\frac{\partial (G\circ F)}{\partial s} =[JM(G)\circ F]\cdot \left[\begin{matrix} 0\\
	\frac{\partial F_1}{\partial s} \\
	\vdots \\
	\frac{\partial F_{n-1}}{\partial s} \\
	ps^{p-1}
	\end{matrix}\right]$. Hence, $\frac{\partial G}{\partial z_n}\circ F =[DG_{n-1}\circ F]\cdot \left[\begin{matrix}
	-\frac{\dot{x}_{1,s}}{ps^{p-1}} \\
	\vdots \\
	-\frac{\dot{x}_{n-1,s}}{ps^{p-1}}
	\end{matrix}\right]$, where $\dot{x}_{j,s}=\frac{\partial F_j}{\partial s}$,$\forall j\in\{1,...,n-1\}$.
	
	The functions $\frac{\dot{x}_{i,s}}{ps^{p-1}}$ are smooth on the normalization of $\cX$, which is $\bC\times \bC$, since the order of vanishing of $\frac{\partial F_i}{\partial s}$ is greater than $p-1$ in $s$. This shows that $\frac{\partial G}{\partial z_n}$ is in the integral closure of $DG_{n-1}$, hence $DG_{n-1}$ is a reduction of $JM_z(\cX)$. 
		
	
	Let $c(\frac{\partial G}{\partial t})$ be the vector field associated to the Cramer's rule in order to solve the equation $$\left[\frac{\partial G}{\partial t}\right]=[DG_{n-1}]\cdot\xi. $$
	
	Since $\frac{\partial G}{\partial t}\circ F =[DG_{n-1}\circ F]\cdot \left[\begin{matrix}
	-\frac{\partial F_1}{\partial t} \\
	\vdots \\
	-\frac{\partial F_{n-1}}{\partial t}
	\end{matrix}\right]$ then 
	$c(\frac{\partial G}{\partial t})\circ F =\left[\begin{matrix}
	-\frac{\partial F_1}{\partial t} \\
	\vdots \\
	-\frac{\partial F_{n-1}}{\partial t}
	\end{matrix}\right] $.
	
	Therefore, $\frac{\partial F_1}{\partial t},...,\frac{\partial F_n}{\partial t}$, when viewed as functions on $\cX$, are Lipschitz functions rel $Y$ if and only if $c(\frac{\partial G}{\partial t})$ is a Lipschitz vector field rel $Y$ on $\cX$.  This is equivalent to $(\frac{\partial G}{\partial t})_D\in\overline{(DG_{n-1})_{D,Y}}$, by Proposition \ref{P2.24}, since $\frac{\partial G}{\partial t}\subset \overline {m_YJM_z(\cX)}=\overline {m_YDG_{n-1}}$.
\end{proof}

\newpage

{\sc Terence James Gaffney
	
	\vspace{0.5cm}
	
	{\small Department of Mathematics 
		
		Northeastern University
		
		{\tiny 567 Lake Hall - 02115 - Boston - MA, USA, t.gaffney@neu.edu }}}

\vspace{0.7cm}	
	
{\sc Thiago Filipe da Silva

\vspace{0.5cm}
	
	{\small Department of Mathematics
	
	Federal University of Esp\'irito Santo
	
	{\tiny Av. Fernando Ferrari, 514 - Goiabeiras, 29075-910 - Vit\'oria - ES, Brazil, thiago.silva@ufes.br}}}

\end{document}